\documentclass[12pt, a4paper]{amsart}
\usepackage{amssymb}
\usepackage{amsmath,esint,enumitem}
\usepackage{colortbl}
\usepackage[dvipsnames]{xcolor}
\usepackage[colorlinks, allcolors=Fuchsia,pdfstartview=,pdfpagemode=UseNone]{hyperref} 
\usepackage{graphics,fullpage}
\usepackage{multirow}
\usepackage[normalem]{ulem}
\usepackage{tikz}
\makeatletter
\@namedef{subjclassname@2020}{\textup{2020} Mathematics Subject Classification}
\makeatother

\newtheorem{theorem}[equation]{Theorem}
\newtheorem{lemma}[equation]{Lemma}
\newtheorem{proposition}[equation]{Proposition}
\newtheorem{corollary}[equation]{Corollary}

\theoremstyle{definition}
\newtheorem{definition}[equation]{Definition}

\theoremstyle{remark}
\newtheorem{remark}[equation]{Remark}
\numberwithin{equation}{section}

\let\oldmarginpar\marginpar
\renewcommand\marginpar[1]{\-\oldmarginpar[\raggedleft\footnotesize #1]%
{\raggedright\footnotesize #1}}

\newcommand{\supp}{\operatorname{supp}}

\newcommand{ \R }{ \mathbb{R} }
\newcommand{ \Rn }{ {\mathbb{R}^n} }

\newcommand{\Phiw}{{\Phi_{\mathrm{w}}}}
\newcommand{\Phic}{{\Phi_{\mathrm{c}}}}

\newcommand{\tphi}{{\tilde\phi}}

\newcommand{\tpsi}{{\tilde\psi}}

\newcommand{\diam}{\operatorname{diam}}

\renewcommand{\epsilon}{\varepsilon}
\renewcommand{\phi}{\varphi}
\renewcommand{\le}{\leqslant}
\renewcommand{\ge}{\geqslant}
\renewcommand{\leq}{\leqslant}
\renewcommand{\geq}{\geqslant}

\newcommand{\loc}{\mathrm{loc}}

\newcommand{\ainc}[1]{\hyperref[ainc]{{\normalfont(aInc){\ensuremath{_{#1}}}}}}
\newcommand{\adec}[1]{\hyperref[adec]{{\normalfont(aDec){\ensuremath{_{#1}}}}}}
\newcommand{\adeci}[1]{\hyperref[adeci]{{\normalfont(aDec){\ensuremath{_{#1}^\infty}}}}}
\newcommand{\inc}[1]{\hyperref[inc]{{\normalfont(Inc){\ensuremath{_{#1}}}}}}
\newcommand{\dec}[1]{\hyperref[dec]{{\normalfont(Dec){\ensuremath{_{#1}}}}}}
\newcommand{\azero}{\hyperref[azero]{{\normalfont(A0)}}}
\newcommand{\aone}{\hyperref[aone]{{\normalfont(A1)}}}
\newcommand{\aonew}{\hyperref[aonew]{{\normalfont(A1)$_{\omega}$}}}
\newcommand{\aonewcomma}{\hyperref[aonewcomma]{{\normalfont(A1$')_{\omega}$}}}

\newcommand{\DMA}[1]{\hyperref[DMA]{{\normalfont(DMA1)\ensuremath{_{#1}}}}}

\newcommand{\aonen}[1]{\hyperref[aone]{{\normalfont(A1-\ensuremath{#1})}}}
\newcommand{\VAn}[1]{\hyperref[VA]{{\normalfont(VA1-\ensuremath{#1})}}}
\newcommand{\wVAn}[1]{\hyperref[wVA]{{\normalfont(wVA1-\ensuremath{#1})}}}
\newcommand{\VMAn}[1]{\hyperref[VMA]{{\normalfont(VMA1-\ensuremath{#1})}}}
\newcommand{\wVMAn}[1]{\hyperref[wVMA]{{\normalfont(wVMA1-\ensuremath{#1})}}}
\newcommand{\DMAn}[1]{\hyperref[DMA]{{\normalfont(DMA1-\ensuremath{#1})}}}
\newcommand{\wDMAn}[1]{\hyperref[wDMA]{{\normalfont(wDMA1-\ensuremath{#1})}}}
\newcommand{\SAn}[1]{\hyperref[SA]{{\normalfont(SA1-\ensuremath{#1})}}}

\begin{document}

\title{Self-improving property for certain degenerate functionals with generalized Orlicz growth}

\author{Vertti Hietanen}
\address{Department of Mathematics and Statistics, FI-00014 University of Helsinki, Finland}
\email{vertti.hietanen@helsinki.fi}

\author{Mikyoung Lee}
\address{Department of Mathematics and Institute of Mathematical Science, Pusan National University, Busan 46241, Republic of Korea}
\email{mikyounglee@pusan.ac.kr}
%
\thanks{}

\date{\today}
\subjclass[2020]{49N60, 35J62, 35B65, 46E35}
\keywords{quasiminimizer,
nonstandard growth, 
Muckenhoupt weight,
generalized Orlicz space,
variable exponent, double phase}

\begin{abstract}
We investigate a self-improving property of variational integrals in a weighted framework under generalized Orlicz growth conditions. Assuming that the weight belongs to an appropriate Muckenhoupt
class and the growth function satisfies standard structural conditions, we prove that
the gradient of any local quasiminimizer has local higher integrability.  
In addition, we establish the existence of minimizers for the associated functional.
\end{abstract}

\maketitle


\section{Introduction}\label{sect:intro}
In this paper, we are devoted to studying the local quasiminimizers of the weighted minimization problem
\[
\min_{v \in W^{1,1}(\Omega)}\int_{\Omega} \phi(x,|\nabla v|)\omega (x)\,dx,
\]
where $\omega$ is a weight function and $\phi$ satisfies generalized Orlicz type growth conditions; see Section~\ref{sect:preliminaries} for the precise assumptions.

In recent decades, variational integrals involving nonstandard growth structures have emerged as a significant topic in the study of nonlinear partial differential equations and calculus of variations; see, for example, \cite{AceMin01,DeFilMin21,DieSchVer09,EspLeoMin04,HasO22,KriMin05}. 
A classical prototype of such problems is the non-autonomous minimization problem 
\[
\min_{v \in W^{1,1}(\Omega)}\int_{\Omega} F(x,\nabla v)\,dx.
\]
Such functionals were extensively studied by Marcellini~\cite{Mar89,Mar91}, who made fundamental contributions to the theory of variational integrals with nonstandard $(p,q)$-growth conditions
\[
|t|^p - 1 \lesssim F(x,t) \lesssim |t|^q +1, \  \ p < q. 
\] 
Zhikov~\cite{Zhi86, Zhi95} considered representative forms of nonstandard variational integrals that serve as prototype models for anisotropic materials and illustrate the occurrence of the Lavrentiev phenomenon.  Among the examples presented in~\cite{Zhi95} are two fundamental classes of integrands. 
The first one takes the form
\[
F(x,t) \approx |t|^{p(x)}, \qquad 1 < \inf p \le \sup p < \infty,
\]
which represents the variable exponent case, where the growth exponent varies with the spatial position. 
Such formulations have found applications in several contexts, most notably in the study of electrorheological fluids~\cite{Ruz00} and image processing models~\cite{CheLR06,HarHLT13}, where material properties or regularization weights depend on location. 
The second class is described by
\[
F(x,t) \approx |t|^{p} + a(x)|t|^{q}, \qquad 1 < p \le q < \infty, \ a(\cdot) \ge 0,
\]
and is commonly referred to as the double-phase case. 
This model captures the mechanical response of composite or inhomogeneous media, in which the coefficient $a(x)$ distinguishes regions with different growth behavior, thus allowing for phase transitions or discontinuities across interfaces.
A comprehensive regularity theory for these problems has been developed in a series of works by Baroni, Colombo, and Mingione \cite{BarCM16, BarColMin18,ColMin15,ColM15b,ColM16}.
Moreover, a modified version of the double phase functional 
\[
F(x,t) \approx (|t|-1)^p_+ + a(x) (|t|-1)^q_+,
\]
was analyzed in \cite{CloGHP_pp18,CupGGP18}, where $(s)_+ : = \max \{ s, 0\}$. 
This functional exhibits a degenerate behavior for small positive values of the gradient. Further studies have also focused on borderline energies such as 
\[
F(x,t) \approx |t|^{p(x)}\log(e +|t|)  \ \text{ and } F(x,t) \approx |t|^p+a(x) |t|^p \log(e +|t|),
\]
see, for instance, \cite{BarCM16,ByunOh17,GiaP13,Ok16}. 
All of these cases are contained in the general setting considered in this paper.


Parallel to these developments, the regularity theory for minimizers and quasiminimizers associated with variational integrals under nonstandard growth conditions has been extensively studied in recent years; see, e.g., \cite{ByunBaa25, DeFilMin231, HHL21, HasO22, KarLee22}, the survey articles \cite{Min06, MinRad21}, and the references therein.
In the unweighted setting, the higher integrability of the gradient of minimizers has been established for several special cases, including the standard $p$-growth case, Orlicz case, the variable exponent case, and the double phase growth case; see, e.g., \cite{BarColMin18, ColMin15, DieSch14,EspLeoMin99, FanZha00, FusSbo90}.
Recently, Harjulehto, H\"ast\"o, and Karppinen \cite{HHKa18} proved local higher integrability for quasiminimizers under generalized Orlicz growth conditions, which unify and extend these classical frameworks. 
Later, Karppinen \cite{Kar21} established global higher integrability for minimizers of the obstacle problem under the same type of growth condition.
In contrast, the corresponding weighted case remains open.
Very recently, one of the authors established the boundedness of the Hardy–Littlewood maximal operator in weighted generalized Orlicz spaces in \cite{Hie25}, finding a sufficient condition on the weight related to the structural condition imposed on the generalized Orlicz function.
Building on this result, we established a Poincar\'e-type inequality in the setting of these spaces and proved the existence of minimizers for the associated functional.
This, in turn, guarantees the existence of the local quasiminimizers considered in our main result.

The aim of this paper is to extend the higher integrability result of Harjulehto, H\"ast\"o, and Karppinen \cite{HHKa18}  to the weighted setting. More precisely, we prove that if $\omega$ belongs to an appropriate Muckenhoupt class, then for any local quasiminimizer $u$ of the corresponding $\psi$-energy, where $\psi(x,t) = \phi(x,t)\omega(x)$, the gradient $\nabla u$ possesses local higher integrability. This result covers, as special cases, the weighted variable exponent growth case and the weighted double phase growth case.

The following is our main result:
\begin{theorem}\label{mainthm}
Let $\Omega \subset \Rn$ be a bounded domain. Let $\omega \in A_p$ and  $\phi \in \Phiw(\Omega)$  satisfy conditions 
\aonew{}, \ainc{p}, and \adeci{q} with $1<p\le q <\infty$. 
 Let $u\in W^{1,\phi}_{\omega,\loc}(\Omega)$ be a local quasiminimizer of
\[
\int_{\Omega} \phi(x, |\nabla u|)\omega(x)\,dx.
\] 
  Then there exists $\epsilon>0$ depending on $n,K,p,q,\beta_0, \beta_1,L,$ and $[\omega]_{A_p}$ such that $\phi(\cdot,|\nabla u|) \in L^{1+\epsilon}_{\omega,\loc}(\Omega)$.
Moreover, 
we have    \[
\left(\frac{1}{\omega(B_r)}\int_{B_r} \phi(x,|\nabla u|)^{1+\epsilon} \omega(x)dx \right)^{\frac{1}{1+\epsilon}}
\lesssim  \frac1{\omega(B_{2r})}\int_{B_{2r}} \phi(x,| \nabla u|)\omega (x)\,dx +1
\] 
for any $B_{2r}\Subset \Omega$ with
$\int_{B_{2r}}\phi(x,|\nabla u|)\omega (x)\,dx \le 1 $.
Here, the implicit constant depends on  $n,K,\beta_0, \beta_1,p,q,L,$ and $[\omega]_{A_p}$. \end{theorem}

Instead of the widely used the condition \aone{} (see
Definition~\ref{def:phicondi}), we assume a continuity condition adapted to the weighted setting, denoted by \aonew{}. The two conditions coincide when $\omega \approx 1$.
For instance, in the variable exponent case $\phi(x,t) = t^{p(x)}$, the condition \aonew{} is satisfied when $1/p$ is globally $\log$-Hölder continuous and $\omega \in A_\infty$. In the double phase case $\phi(x,t) = t^{p} + a(x)t^{q}$, the condition \aonew{} holds if and only if  
$a(x) \lesssim a(y) + \omega(B)^{\frac{q-p}{p}}$
for every $x, y \in B$ (see \cite{Hie25} for further details).

We distinguish the function $\phi$ from the weight $\omega$ by means of the condition \azero{} (see
Definition~\ref{def:phicondi}), since the product  $\psi(x,t) := \phi(x,t)\,\omega(x)$
also defines a generalized Orlicz function. The condition \azero{} roughly states that there exists a constant $ c > 0 $ such that $\phi(x, c) \approx 1$ for almost every $x \in \Omega$, meaning that $\phi$ is essentially \emph{unweighted}. 
Following the same argument as in \cite[Proposition~3.3]{HHS24}, one can show that the condition \aonew{} implies the condition \azero{}  when the domain $\Omega$ is bounded. Hence, although both conditions \azero{} and \aone{} were assumed in the unweighted setting in \cite{HHKa18}, in our main result it suffices to assume only the condition \aonew{}.

The key idea of our proof is to combine a Sobolev–Poincar\'e type inequality and a Caccioppoli inequality adapted to the weighted setting, and to apply a version of Gehring’s lemma for metric spaces with a doubling measure. 
In contrast to the unweighted case, the main challenge lies in proving an appropriate modular-type Sobolev–Poincar\'e inequality compatible with the weight, for which we employ the structural conditions on the generalized Orlicz function $\phi$ together with the properties of Muckenhoupt weights.


\section{Preliminaries and notation}\label{sect:preliminaries}

Let $\Omega$ be a bounded domain in $\Rn$ with $n \ge 2$. 
For $x_0\in \Rn$ and $r>0$, we write $B_r(x_0)$ for the open ball centered at $x_0$ with radius $r$. When the center is clear from the context, we simply write $B_r$ in place of $B_r(x_0)$. 
Given a set $E\subset \Rn$, its usual \textit{characteristic function}  $\chi_E$ is defined by 
$\chi_E(x)=1$ if $x\in E$ and $\chi_E(x)=0$ if $x\not\in E$. 
For $p \in [1,\infty]$, 
we denote its H\"older conjugate by  $p'=\frac{p}{p-1}$. 
The symbol $c>0$ without subscript denotes a generic constant, possibly changing from line to line.

Let $f, g : E \to \R$ be measurable functions on a set  $E\subset\Rn$. For a measurable set $E$ with $0<|E|<\infty$, we denote the average of $f$ over $E$ by 
$(f)_E:=\fint_E f\, dx := \frac{1}{|E|}\int_E f \,dx$. 
The notation $f\lesssim g$ means that
there exists a constant $C>0$ such that $f(y)\le Cg(y)$ all $y\in E$ and 
$f\approx g$ means that $f\lesssim g\lesssim f$. 
When $E\subset \R$, we say that $f$ is \textit{almost increasing} on $E$ with constant 
$L\ge 1$ if $f(s)\le L f(t)$ whenever $s,t \in E$ and $s\le t$. If $L$ can be chosen as $1$, we say that  
$f$ is \textit{increasing} on $E$.
The notations of \textit{almost decreasing} and \textit{decreasing} are defined in the obvious analogous way. 

%

\subsection{Generalized Orlicz functions}
We begin by stating the basic conditions on the energy function $\phi:\Omega\times[0,\infty)\to [0,\infty)$. 
We define its (left-continuous) inverse function of $\phi$ with respect to $t$ by
\[
\phi^{-1}(x,t):= \inf\{\tau\geq 0: \phi(x,\tau)\geq t\}.
\]
If $\phi$ is strictly increasing and continuous in $t$, then $\phi^{-1}$ coincides with the usual inverse function.
For the  following definitions and properties, 
we refer to \cite[Chapter~2]{HarH19}. 
Our initial condition concerns the regularity of $\phi$ with respect to its second variable,  
which is required to hold for every $x\in \Omega$ with
a constant $L\ge 1$ independent of $x$. 

We denote a weight by $\omega$, which is a non-negative and locally integrable function. 

\begin{definition} \label{def:phicondi}
Let $\phi:\Omega\times[0,\infty)\to [0,\infty)$ and $\gamma\in\R$.
We say that $\phi$ satisfies
\vspace{0.2cm}
\begin{itemize}
\item[\normalfont(aInc)$_\gamma$]\label{ainc} if
$t\mapsto t^{-\gamma}\phi(x,t)$ is almost increasing on $(0,\infty)$ with constant $L$;
\vspace{0.2cm}
\item[\normalfont(Inc)$_\gamma$]\label{inc} if
$t\mapsto t^{-\gamma}\phi(x,t)$ is increasing on $(0,\infty)$;
\vspace{0.2cm}
\item[\normalfont(aDec)$_\gamma$]\label{adec} if 
$t\mapsto t^{-\gamma}\phi(x,t)$ is almost decreasing on $(0,\infty)$ with constant $L$;
\vspace{0.2cm}
\item[\normalfont(aDec)$_\gamma^\infty$] \label{adeci}
if
$t \mapsto t^{-\gamma}(\phi(x,t)+1)$ is almost 
decreasing in $(0,\infty)$ with constant $L$;
\vspace{0.2cm}
\item[\normalfont(Dec)$_\gamma$]\label{dec} if
$t\mapsto t^{-\gamma}\phi(x,t)$ is decreasing on $(0,\infty)$;
\vspace{0.2cm}
\item[\normalfont(A0)] \label{azero} 
if there exists $\beta_0 \in (0,1]$ such that $\beta_0 \le \phi^{-1}(x, 1) \le \frac{1}{\beta_0}$ for almost every $x\in \Omega$;
\item[\normalfont(A1)$_{\omega}$] \label{aonew} 
if there exists $\beta_1 \in (0,1]$ such that $\beta_1 \phi^{-1}(x,t)\le \phi^{-1}(y, t)$ for every $t \in \big[ 1, \frac1{\omega(B)}\big]$, almost every $x, y \in B \cap \Omega$ and every ball $B$ with $\omega(B) \le 1$; 
\item[\normalfont(A1$')_\omega$]\label{aonewcomma}
if there exists $\hat{\beta}_1 \in (0,1]$ such that $ \phi(x,\hat{\beta_1}t)\le \phi(y, t)$ for every $\phi(y,t) \in \big[ 1, \frac1{\omega(B)}\big]$, almost every $x, y \in B \cap \Omega$ and every ball $B$ with $\omega(B) \le 1$; 
\item[\normalfont(A1)]\label{aone}
if there exists $\tilde{\beta}_1 \in (0,1]$ such that $\tilde{\beta}_1 \phi^{-1}(x,t)\le \phi^{-1}(y, t)$ for every $t \in \big[ 1, \frac1{|B|}\big]$, almost every $x, y \in B \cap \Omega$ and every ball $B$ with $|B| \le 1$.
\end{itemize}

We say that $\phi$ satisfies \ainc{}, \adec{} or \adeci{} if it satisfies \ainc{\gamma}, \adec{\gamma} or \adeci{\gamma}, respectively, for some $\gamma > 1$.
The 
condition \adeci{\gamma} intuitively means that $t \mapsto t^{-\gamma}\phi(x,t)$ is 
almost increasing for $t>T$ for some constant $T>0$. 
\end{definition}

For $p,q>0$, the condition \ainc{p} (resp. \adec{q}) on $\phi$ with constant $L\geq 1$ is equivalent to 
the inequality 
\[
\phi(x,\lambda t)\le L\lambda^p\phi(x,t)
\quad( \text{resp. }
 \phi(x,\Lambda t) \le L \Lambda^q \phi(x,t) )
\]
for all $(x,t)\in \Omega\times [0,\infty)$ and $0\le \lambda\le 1 \le \Lambda$. 
In addition, \ainc{} or \adec{} 
are equivalent to $\nabla_2$-condition or $\Delta_2$-condition, respectively \cite[Chapter 2]{HarH19}.  
%

Let $\phi,\psi : [0,\infty)\to [0,\infty)$ be increasing functions  such that $\phi$ satisfies \ainc{1} and \adec{q} for some $q\ge 1$, and $\psi$ satisfies \adec{1}. 
According to \cite[Lemma~2.2.1]{HarH19}, there exist a convex function $\tphi$ and a concave function $\tpsi$ such that $\phi\approx\tilde \phi$ and $\psi\approx \tpsi$. 
Applying Jensen's inequality to $\tphi$ and $\tpsi$ then gives  
\begin{equation}\label{Jensen}
\phi\bigg(\fint_\Omega |f|\,dx\bigg) \lesssim \fint_\Omega \phi(|f|)\,dx 
\quad\text{and}\quad 
\fint_\Omega \psi(|f|)\,dx \lesssim \psi\bigg(\fint_\Omega |f|\,dx\bigg)
\end{equation}
for every $f\in L^1(\Omega)$. Here, the
implicit constants depend on the constant $L$ from \ainc{1} and \adec{q} or \adec{1}, based on the constants involved in the equivalences $\phi\approx\tilde \phi$ and $\psi\approx \tpsi$.

\medskip

We introduce the classes of $\Phi$-functions and generalized Orlicz spaces, following \cite{HarH19}. 
While our main interest lies in convex functions that naturally arise in minimization problems and the study of related PDEs,  
the broader class $\Phiw(\Omega)$ is particularly useful  for the approximation of functionals.

\begin{definition}\label{defPhi}
A function $\phi:\Omega\times[0,\infty)\to [0,\infty]$ is said to be a $\Phi$-prefunction if $x\mapsto \phi(x,|f(x)|)$ is measurable for every measurable function $f$ on $\Omega$, 
$t\mapsto \phi(x,t)$ is increasing, 
and $\displaystyle \phi(x,0)=\lim_{t\to0^+}\phi(x,t)=0$ and $\displaystyle\lim_{t\to\infty}\phi(x,t)=\infty$
 for every $x\in\Omega$.
 
The $\Phi$-prefunction $\phi$ is said to be 
\begin{itemize}
\item[(1)] a \textit{weak $\Phi$-function}, denoted $\phi\in\Phiw(\Omega)$, if it satisfies \ainc{1};
\item[(2)] a \textit{convex $\Phi$-function}, denoted $\phi\in\Phic(\Omega)$, if $t\mapsto \phi(x,t)$ is left-con\-tin\-u\-ous and convex for every $x\in\Omega$.
\end{itemize}
The subsets of $\Phiw(\Omega)$ and $\Phic(\Omega)$ consisting of functions independent of the first variable, i.e., functions of the form $\phi(x,t)=\phi(t)$, are denoted by 
$\Phiw$ and $\Phic$, respectively.
\end{definition}
Since convexity implies \inc{1}, we have that $\Phic(\Omega)\subset \Phiw(\Omega)$.

The following lemma indicates that one can upgrade \adeci{} to \adec{}  while preserving several other properties. We remark that our main results proved under the assumption \adec{q} remain valid also in the case of \adeci{q} by applying the argument in the following lemma. Although the lemma originally considered the assumption \aone{} in part \textit{(3)}, the proof is essentially the same for \aonew{}.
 \begin{lemma}[\cite{HHL21}, Lemma~2.3]\label{lem:phi+t}
Let $\phi \in \Phiw(\Omega)$ and define $\psi(x,t):= \phi(x,t) +t$. Then $\psi \in \Phiw(\Omega)$. Moreover,
\begin{enumerate}
\item[(1)] if $\phi$ satisfies \azero{}, then $\phi \le \psi \lesssim \phi +1$ and $\psi$ satisfies \azero{}; 
\item[(2)] if $\phi$ satisfies \adeci{q} and \azero{}, then $\psi$ satisfies \adec{q};
\item[(3)] if $\phi$ satisfies \aonew{}, then $\psi$ satisfies \aonew{}.
\end{enumerate}
\end{lemma}

The following lemma shows the relation between conditions \azero{}, \aonew{} and \aonewcomma{} in a bounded domain. The first part follows together from \cite[Proposition~3.3]{HHS24}, and \cite[Proposition~4.1.5]{HarH19}. The second part follows from \cite[Corollary~4.1.6]{HarH19}. Again, these results consider the \aone{} condition, but the proofs are similar for \aonew{}.

\begin{lemma}\label{lem:a0a1}
Let $\phi \in \Phiw(\Omega)$ and $\Omega\subset \Rn$ be bounded. Then
\begin{enumerate}
\item[(1)] if $\phi$ satisfies \aonew{}, then it satisfies  \azero{} and \aonewcomma{};
\item[(2)] if $\phi$ satisfies \azero{} and \aonewcomma{}, then it satisfies \aonew{}. 
\end{enumerate}
Here, the implied constants depend on the constants of the assumed conditions and $L$.
\end{lemma}

\subsection{Generalized Orlicz spaces}
We denote by $L^0(\Omega)$ the set of measurable
functions in $\Omega$.
For $\phi\in\Phiw(\Omega)$, the \textit{generalized Orlicz space} (also called the \textit{Musielak--Orlicz space}) is defined as
\[
L^{\phi}(\Omega):=\big\{f\in L^0(\Omega):\|f\|_{L^\phi(\Omega)}<\infty\big\},
\] 
with the (Luxemburg) norm 
\[
\|f\|_{L^\phi(\Omega)}:=\inf\bigg\{\lambda >0: \varrho_{L^\phi(\Omega)}\Big(\frac{f}{\lambda}\Big)\leq 1\bigg\},
\quad\text{where}\quad\varrho_{L^\phi(\Omega)}(f):=\int_\Omega\phi(x,|f(x)|)\,dx .
\] 
We define $W^{1,\phi}(\Omega)$ as the set of functions $f\in L^\phi(\Omega)\cap L^1_{\loc}(\Omega)$ 
such that $\|f\|_{W^{1,\phi}(\Omega)}:=\|f\|_{L^\phi(\Omega)}+\big\|| \nabla f|\big\|_{L^\phi(\Omega)}<\infty$. 
If $\phi$ satisfies \adec{}, then $f\in L^\phi(\Omega)$ if and only if 
$\varrho_{L^\phi(\Omega)}(f)<\infty$. 
When $\phi$ satisfies \azero{}, \ainc{} and \adec{},
 the spaces $L^\phi(\Omega)$ and $W^{1,\phi}(\Omega)$ are both reflexive Banach spaces.
We denote by $W^{1,\phi}_0(\Omega)$ the closure of 
$C^\infty_0(\Omega)$ in $W^{1,\phi}(\Omega)$.

For   
$\phi\in\Phiw(\Omega)$, the \textit{weighted generalized Orlicz space} is defined as
\[
L_\omega^{\phi}(\Omega)=L^{\phi}(\Omega,\omega):=\big\{f\in L^0(\Omega):\|f\|_{L_\omega^\phi(\Omega)}<\infty\big\},
\] 
with the (Luxemburg) norm 
\[
\|f\|_{L_\omega^\phi(\Omega)}:=\inf\bigg\{\lambda >0: \varrho_{L_\omega^\phi(\Omega)}\Big(\frac{f}{\lambda}\Big)\leq 1\bigg\},
\quad\text{where}\quad\varrho_{L_\omega^\phi(\Omega)}(f):=\int_\Omega\phi(x,|f(x)|)\omega (x)\,dx.
\] 
A typical example of the generalized Orlicz function $\phi$  is
$\phi(x,t) = t^{\gamma}$ with $\gamma>1$.
In this case, the weighted generalized Orlicz space $L^{\phi}_{\omega}(\Omega)$ coincides with the weighted Lebesgue space $L^{\gamma}_{\omega}(\Omega)$.
 For further details on generalized Orlicz 
and Orlicz--Sobolev spaces, we refer to the monographs 
\cite{CheGSW21, HarH19} and to \cite[Chapter~2]{DieHHR11}.

  \subsection{Muckenhoupt weight}
A weight $\omega$ belongs to the Muckenhoupt class $A_{\gamma}$, denoted by $\omega \in A_{\gamma}$, for $1<\gamma<\infty$, 
if 
$$
[\omega]_{A_\gamma} :=\sup_{B \subset \Rn} \left( \fint_B \omega (x) dx \right) \left( \fint_B \omega (x)^{-\frac{1}{\gamma-1}} dx \right)^{\gamma-1} < \infty,
$$  
where the supremum is taken over all balls $B \subset \Rn$.
We denote  
$$w(\Omega)=\int_{\Omega} \omega (x) dx.$$
There is an alternate way of defining the $A_\gamma$-weight. For any nonnegative integrable function $f$ and any ball $ B\subset \Rn,$ $\omega \in A_\gamma$ with $1< \gamma<\infty$ if and only if there holds 
\begin{equation}\label{def:eqAp}
    \left(\fint_B |f(x)|\, dx\right)^{\gamma}\leq \frac{[\omega]_{A_\gamma}}{\omega(B)} \int_B|f(x)|^\gamma\omega(x) \, dx.
\end{equation} 
Each $A_\gamma$-weight is doubling; that is, there exists a constant $c=c(n, \gamma, [\omega]_{A_{\gamma}})>0$ such that $\omega(B_{2r}(y)) \le c\,\omega(B_r(y))$ for every $y\in \Rn$ and $r>0.$

The following properties for the Muckenhoupt weights $\omega\in A_\gamma$ are well-known; see \cite{Gr09} for more details with their proofs.
\begin{lemma}\label{lem:property_A}
	Let $\omega \in A_\gamma$ with $1<\gamma<\infty$. 
	The following properties hold:
	\begin{itemize}
		\item[(1)] For $\gamma_1\leq \gamma_2$,
\[
		A_{\gamma_1} \subseteq A_{\gamma_2}.
		\]
		\item[(2)]
		There exists a small constant $\tilde \varepsilon  >0$ depending on $n,\gamma,$ and $[\omega]_{A_\gamma}$ such that 
		\[
		A_\gamma \subset A_{\gamma-\tilde \varepsilon}.		\]
		\item[(3)] For measurable sets $C, D$ with $C \subset D \subset \mathbb{R}^n$,
		there exists $\sigma  \in (0,1)$ depending on $n,\gamma,$ and $[\omega]_{A_\gamma}$ such that
		\[
		[\omega]_{A_\gamma}^{-1}\left(\frac{|C|}{|D|}\right)^\gamma\leq  \frac{\omega(C)}{\omega(D)} \leq c \left(\frac{|C|}{|D|} \right)^{\sigma}
		\]
		where a constant $c>0$ depends only on $n$ and $\gamma$. 
	\end{itemize}
\end{lemma}

%

The following lemma provides a weighted modular type inequality for $A_p$ weights.
\begin{lemma}\label{lem:wjens}
    Let $\varphi:[0,\infty) \to [0,\infty]$ be a $\Phi$-prefunction that satisfies \ainc{p} with $p \in (1,\infty)$. If $\omega \in A_p$, then there exists $\beta=\beta(L,p,[\omega]_{A_p})>0$ such that the following inequality holds for every ball $B\subset \Rn$ and every $f\in L^1(B)$:
\begin{equation*}
    \varphi \left(\beta \fint_B |f(x)|dx\right)\leq \frac{1}{\omega(B)}\int_B \varphi(|f(x)|)\omega(x)dx.
\end{equation*}
\end{lemma}

\begin{proof}
    Since $\varphi(t)$ satisfies \ainc{p}, we have that $\varphi_p(t):=\varphi(t^\frac{1}{p})$ satisfies \ainc{1}. Then by \cite[Lemma 2.2.1]{HarH19} there exists convex and increasing  $\tilde{\phi}$ with $ \varphi_p(\frac{1}{\alpha}t) \leq \tilde{\phi}(t) \leq \varphi_p(\alpha t) $ for some $\alpha=\alpha(L)\geq 1$. By \eqref{def:eqAp} and Jensen's inequality for $\tilde{\phi}$ with measure $d\mu(x)=\omega(x)dx$,
    \begin{align*}
        \varphi\left(\frac{1}{[\omega]_{A_p}^{\frac{1}{p}}\alpha^{\frac{2}{p}}} \fint_B |f(x)|dx \right)
        &\leq \tilde{\phi}\left(\frac{1}{[\omega]_{A_p}\alpha}\left(\fint_B |f(x)|dx\right)^p\right) \\
        &\leq \tilde{\phi}\left(\frac{1}{\omega(B)}\int_B \frac{1}{\alpha}|f(x)|^p\omega(x)dx\right) \\
        &\leq \frac{1}{\omega(B)}\int_B \tilde{\phi}\bigg(\frac{1}{\alpha}|f(x)|^p\bigg)\omega(x)dx \\
        &\leq \frac{1}{\omega(B)}\int_B \varphi(|f(x)|)\omega(x)dx. \qedhere
    \end{align*}
\end{proof}

We end this section by presenting a metric space version of Gehring’s lemma with a doubling measure, which will be used later in the proof of our main theorem; see \cite[Theorem~3.3]{ZG05} for the proof and further details.
\begin{lemma}\label{lem:Gehring}
The measure $\mu$ is assumed to be doubling, that is, there exists a constant $c_d \ge 1$ such that 
\[
0<\mu (B_{2r}) \le c_d\, \mu (B_r) <\infty
\]
for every ball $B_{2r} \Subset \Omega$.
Let $f \in L^1_{\loc}(\Omega,\mu)$ and $g \in L^{\sigma}_{\loc}(\Omega,\mu), \ \sigma >1$, be non-negative functions and let $\lambda >1$. Assume that there exist a constant $c_1$ and an exponent $0<m<1$ such that 
\[
\fint_{B_R} f \,d\mu \le c_1 \bigg( \bigg[ \fint_{B_{\lambda R}} f^m \,d\mu \bigg]^{\frac1m} + \fint_{B_{\lambda R}} g \,d\mu \bigg)
\]
for every ball $ B_{\lambda R}\Subset \Omega.$
Then
there exist a constant $c_2$ and an exponent $\epsilon>0$ depending on $c_d, c_1$ and  $\lambda$ such that 
\[
\bigg(\fint_{B_R}f^{1+\epsilon}\,d\mu \bigg)^{\frac1{1+\epsilon}} \le c_2 \bigg( \fint_{\lambda R} f \,d\mu + \bigg[\fint_{B_{\lambda R}} g^\sigma \, d\mu \bigg]^{\frac1\sigma}\bigg).
\]
\end{lemma}

\section{Existence}

Let $\Omega$ be a bounded domain in $\Rn$. Given a weight $\omega$, letting $\psi(x,t) := \phi(x,t)\omega (x)$, we note that if $\phi \in  \Phiw(\Omega)$, then $\psi \in  \Phiw(\Omega)$.
Throughout the paper, we are dealing with the following local quasiminimizers of the $\psi$-energy
\[
I_\psi(u):=\varrho_{L^\psi(\Omega)}(\nabla u)= \int_{\Omega} \psi(x, |\nabla u|)\,dx.
\]

\begin{definition}\label{def:quasiminimizer}
Let $\psi \in \Phiw(\Omega)$. 
A function $u \in W_{\loc}^{1,\psi}(\Omega) $ is called a local quasiminimizer of the $\psi$-energy in $\Omega$ 
if there exists a constant $K \ge 1$ such that
\[
\int_{ \{ v \neq 0\} } \psi(x, | \nabla u|)\,dx \le K
 \int_{  \{ v \neq 0\} } \psi(x, | \nabla(u+v)|)\,dx
\]
for every $v \in W^{1,\psi}(\Omega)$ with  $\supp  v := \overline{ \{ v \neq 0\} } \subset \Omega$.

\end{definition}
When $K=1$ in the above definition, we call them (local) minimizers. 
 In this section, we prove the existence of the minimizer of the $\psi$ energy,
following the strategy in \cite[Chapters~6 and 7]{HHK16}.

Let $V$ be a reflexive Banach space and $I : V \to \R.$ The operator $I$ is said to be
convex if $I(tu+(1-t)v) \leq tI(u)+(1-t)I(v)$ for all $t \in [0, 1]$ and $u, v \in V$ .
It is lower semicontinuous if $I(u) \leq \liminf_{i \to \infty} I(u_i)$ whenever $u_i \to u$ in $V,$ and
coercive if $I(u_i) \to \infty$ whenever $\|u_i\|_V\to \infty$. 

We are going to use the following well-known lemma in functional analysis.

\begin{lemma}[\cite{KS80}, Theorem 2.1]\label{lem:Vexistence}
    Let $V$ be a reflexive Banach space. If $I:V \to \R$ is a convex, lower semicontinuous, and coercive operator, then there is an element in $V$ that minimizes $I$.
\end{lemma}

\begin{lemma}[\cite{HarH19}, Lemma 3.7.7 with $\mu=\omega(x)dx$]\label{lem:pqwincl}
    Let $\phi\in \Phiw(\Omega)$  satisfy \azero{}, \ainc{p} and \adec{q} for $1\le p <\infty$ and $1\le q \le \infty$ . Then 
    \[
    L^p(\Omega,\omega)\cap L^q(\Omega,\omega) \hookrightarrow L^\psi(\Omega) \hookrightarrow L^p(\Omega, \omega)+L^q(\Omega, \omega) .   \]
    
\end{lemma}

Since $\Omega$ is a bounded domain in $\Rn$, we note that
$\omega(\Omega)<\infty$ ($\omega$ is locally integrable in $\Rn$) and $ L^p(\Omega, \omega)+L^\infty(\Omega, \omega)=L^p(\Omega, \omega)$. The following is an immediate consequence of Lemma~\ref{lem:pqwincl}:
\begin{corollary}\label{cor:inclphitop}
    Let 
     $\phi\in \Phiw(\Omega)$  satisfy \azero{} and \ainc{p}. Then 
   $L^\psi(\Omega)
     \hookrightarrow L^p(\Omega, \omega)$.
\end{corollary}

\begin{lemma}\label{lem:psitoL1loc}
    Let 
    $\phi\in \Phiw(\Omega)$  satisfy \azero{} and \ainc{p}. If $\omega \in A_p$, then $L^\psi(\Omega) \subset L^1_\loc(\Omega)$.
\end{lemma}
\begin{proof}
Let $K\subset \Omega$ be compact and $f \in L^\psi(\Omega)$. By Corollary \ref{cor:inclphitop}, $f \in L^p(\Omega, \omega)$.
    By Hölder's inequality, we obtain
    \[\int_K|f(y)|\,dy=\int_K |f(y)|\omega(y)^{\frac{1}{p}}\omega(y)^{-\frac{1}{p}}\,dy\leq \left(\int_K |f(y)|^p\omega(y) \, dy\right)^{\frac{1}{p}}\left(\int_K \omega(y)^{-\frac{1}{p-1}}\,dy \right)^{1-\frac{1}{p}}.\]
    Since $\omega \in A_p$, it follows that the dual weight $\omega^{-\frac{1}{p-1}}$ is locally integrable in $\Rn$. Thus $f \in L^1_\loc(\Omega)$.
\end{proof}

By \cite[Theorem~6.1.4]{HarH19}, we obtain the following corollary:
\begin{corollary}\label{cor:Wrefl}
Let 
    $\phi\in \Phiw(\Omega)$ satisfy \azero{}, \ainc{p} and \adec{q} with $1<p\leq q <\infty$.
    Then $W^{1,\psi}(\Omega)$ is reflexive.
\end{corollary}

\begin{lemma}\label{lem:Wphitop}
     Let 
      $\phi\in \Phiw(\Omega)$  satisfy \azero{} and \ainc{p}. Then $W^{1,\psi}(\Omega)\hookrightarrow W^{1,p}(\Omega,\omega)$.
\end{lemma}
\begin{proof}
    By \cite[Lemma~6.1.5]{HarH19}, $\|u\|_{W^{1,\psi}(\Omega)}\approx \|u\|_{L^\psi(\Omega)}+\|\nabla u\|_{L^\psi(\Omega)}$ for $u \in W^{1,\psi}(\Omega)$. Then Corollary \ref{cor:inclphitop} yields the claim.
\end{proof}

It follows by Corollary \ref{cor:Wrefl} that $W_0^{1,\psi}(\Omega)$ is reflexive, since $W_0^{1,\psi}(\Omega)$ is a closed subspace of the reflexive Banach space $W^{1,\psi}(\Omega)$.
\begin{lemma}\label{lem:W0refl}
    Let 
    $\phi\in \Phiw(\Omega)$ satisfy \azero{}, \ainc{p} and \adec{q} with $1<p\leq q <\infty$. If $\omega \in A_p$, then $W_0^{1,\psi}(\Omega)$ is a reflexive Banach space. 
\end{lemma}

When $\omega \in A_p$ and $\Omega$ is bounded, the weighted Sobolev space can be embedded into unweighted space.
\begin{lemma}[\cite{Kil94}, Remark~2.4 with $p=q$]\label{lem:omegatodx}
    Let $\Omega$ be a bounded open set and $\omega^{\frac{1}{1-p}}\in L^1(\Omega)$. Then $W_0^{1,p}(\Omega,\omega) \subset W_0^{1,1}(\Omega)$. 
\end{lemma}

To prove coercivity, we need the Poincar\'e inequality in weighted space.

\begin{theorem}\label{thm:wpoin}
Let $\omega \in A_p$ and $\phi\in \Phiw(\Omega)$  satisfy 
    \aonew{}, \ainc{p} with $p \in (1,\infty)$. Then
    \[
        \|{u}\|_{L^\phi_\omega(\Omega)}\lesssim\diam(\Omega) \|{\nabla u}\|_{L^\phi_\omega(\Omega)}
    \]
    for every $u \in W_0^{1,\psi}(\Omega)$.
\end{theorem}
\begin{proof}
 By Lemma \ref{lem:a0a1}, $\phi$ satisfies \azero{} and \aonewcomma{}. Since $\phi$ satisfies \ainc{p}, we have by Lemmas~\ref{lem:Wphitop} and \ref{lem:omegatodx} that $W_0^{1,\psi}(\Omega)\subset W_0^{1,1}(\Omega)$. If 
    $u\in W_0^{1,1}(\Omega)$, then 
    \[|u(x)|\lesssim \diam(\Omega)M|\nabla u|(x) \quad \text{a.e.}  \] 
    by \cite[Chapter~6]{Boj87}. Here,
     the operator $M$ is a (Hardy-Littlewood) maximal operator which is defined by 
     \[
     Mf(x):=\sup_{ B \ni x }\frac{1}{|B|} \int_{B\cap \Omega}|f(y)|\,dy,
     \]
     where the supremum is taken over all open balls $B$ containing $x$.
      By the boundedness of maximal operator in weighted generalized Orlicz spaces in \cite{Hie25}, we obtain 
    \[
        \|{u}\|_{L^\phi_\omega(\Omega)}\lesssim \diam(\Omega) \|{M|\nabla u|}\|_{L^\phi_\omega(\Omega)} \lesssim \diam(\Omega) \|{\nabla u}\|_{L^\phi_\omega(\Omega)}. \qedhere
    \]
\end{proof}

\begin{theorem}\label{thm:existence}
     Let $\omega \in A_p$ and $\phi\in \Phic(\Omega)$ satisfy 
     \aonew{}, \ainc{p} and \adec{q} with $1<p\leq q<\infty$. Let $z \in W^{1,\psi}(\Omega)$. 
     Then there exists a function $u\in z+ W_0^{1,\psi}(\Omega)$ such that
    \[I_\psi(u)=\inf_{v \in z+ W_0^{1,\psi}(\Omega)}I_\psi(v).\]
\end{theorem}

\begin{proof}
    Let $J(u):=I_\psi(u+z)$, where $u \in W_0^{1,\psi}(\Omega)$. By Lemma \ref{lem:W0refl}, the weighted generalized Orlicz-Sobolev space $W_0^{1,\psi}(\Omega)$ is a reflexive Banach space. 
    
    Now we show that $J$ is convex, lower semicontinuous, and coercive.
The operator $J$ is convex, since $\psi \in \Phic(\Omega)$. By \cite[Lemma~3.1.4]{HarH19} (also see \cite[Theorem~2.1.17]{DieHHR11}), $J$ is also lower semicontinuous. Let $(u_i)$ be a sequence of functions in $W_0^{1,\psi}(\Omega)$. If $\|u_i\|_{W^{1,\psi}(\Omega)} \to \infty$, then by the weighted Poincar\'e inequality (Theorem~\ref{thm:wpoin}) we have $\|\nabla u_i\|_{L^{\psi}(\Omega)} \to \infty$. Therefore, $J(u_i)=\varrho_{L^\psi(\Omega)}(\nabla u_i + \nabla z)\to \infty$ as $i\to \infty$ and so the operator $J$ is coercive. 
Thus, the claim follows from Lemma $\ref{lem:Vexistence}$.
\end{proof}

\begin{theorem}
    Let 
    $\omega \in A_p$. Assume that $\phi\in \Phic(\Omega)$ is strictly convex and satisfies \azero{}. Then the possible minimizing function $u$ is unique up to a set of measure zero.
\end{theorem}

\begin{proof}
    Assume that $u_1$ and $u_2$ are two minimizers of the $\psi$ energy with $|\{\nabla u_1 \neq \nabla u_2\}|>0$. When $\nabla u_1(x) \neq \nabla u_2(x)$, we obtain that 
    \[\psi(x,\tfrac{1}{2}\nabla u_1(x)+\tfrac{1}{2}\nabla u_2(x))<\tfrac{1}{2}\psi(x,\nabla u_1(x))+\tfrac{1}{2}\psi(x,\nabla u_2(x)).\] We set $v=\tfrac{1}{2}(u_1+u_2)$,
    from strict convexity of $\psi$. The previous inequality implies that 
    \[
    I_\psi(v)<\tfrac{1}{2}I_\psi(u_1)+\tfrac{1}{2}I_\psi(u_1)=\inf_{u \in z+ W_0^{1,\psi}(\Omega)}I_\psi(u),
    \]
    which is a contradiction. Therefore, we have that $\nabla u_1 = \nabla u_2$ almost everywhere.
    
    We note that $u_1-u_2\in W_0^{1,\psi}(\Omega)\hookrightarrow W_0^{1,1}(\Omega)$ by Lemmas \ref{lem:Wphitop} and \ref{lem:omegatodx}. Then by Poincar\'e inequality in $W_0^{1,1}(\Omega)$ (see \cite[Theorem~7.44]{GT01}), we obtain that 
    \[
    \|u_1-u_2\|_{L^1(\Omega)}\lesssim \|\nabla u_1-\nabla u_2\|_{L^1(\Omega)}=0,
    \]
which yields that $u_1(x)=u_2(x)$ for a.e $x \in \Omega$.
\end{proof}

\begin{remark}
For $\phi\in \Phiw(\Omega)$, Lemma~2.2.1 in \cite{HarH19} yields that there exists a convex function $\tphi$ such that $\phi\approx\tilde \phi$. Then, Theorem~\ref{thm:existence} ensures the existence of a minimizer 
 of 
\[
 \int_{\Omega} \tilde \phi(x, |\nabla u|)\omega(x)\,dx,
 \]
 which in turn implies the existence of many quasiminimizers of the $\psi$-energy.
\end{remark}

\section{Higher integrability }
\label{sect:comparison}

In this section, we derive higher integrability. 
%
%
First, we need a few estimates for the Sobolev-Poincar\'e type inequality in our weighted setting.
\begin{lemma}\label{lem:wkeyest}
 Let $\omega \in A_p$ and $\phi\in \Phiw(B_r)$ satisfy \aonew{} and \ainc{p}, with $p \in (1, \infty)$.
 Then there exists $\beta=\beta(\beta_1,L,p,[\omega]_{A_p})>0$ such that
\begin{equation}\label{eq:keyeq}
        \phi\left(x,\beta \fint_{B_r} |f| \, dy\right)\leq \fint_{B_r} \phi(y, |f|) \, dy +1
    \end{equation}
    for almost every $x\in B_r$ and $f\in L^{\varphi}(B_r,\omega)$ with 
    $\varrho_{L_\omega^\phi(B_r)}(f)\le 1.$
\end{lemma}

\begin{proof}
    Let $\beta'_J=\beta'_J(L)>0$ be the implicit constant from Jensen's inequality \eqref{Jensen} and $\beta_J=\beta_J(L,p,[\omega]_{A_p})>0$ the constant from Lemma \ref{lem:wjens}. By Lemma \ref{lem:a0a1}, \aonew{} implies that $\phi$ satisfies \aonewcomma{}. Thus, there exists $\hat{\beta}_1=\hat{\beta}_1(\beta_1, L)>0$ such that 
    \begin{equation}\label{eq:A1w'} \phi(x,\hat{\beta}_1 t)\leq \phi(y,t)+1
    \end{equation}
     for every $\phi(y,t)\in[0,\frac{1}{\omega(B_r)}]$
     and almost every $x,y\in B_r$. Denote $\phi_{B_r}^-(t):=\textup{essinf}_{x\in B_r}\phi(x,t)$. Since $\varrho_{L_\omega^\phi(B_r)}(f)\le 1,$
    Lemma \ref{lem:wjens} implies that $\phi_{B_r}^-\left(\frac{\beta_J }{2\max\{1,\beta'_J\} L}\fint_{B_r}f\right)\leq \frac{1}{\omega(B_r)}$. Thus \eqref{eq:A1w'} and Jensen's inequality \eqref{Jensen} give
    \begin{align*}
        \phi\left(x,\frac{\hat{\beta}_1\beta_J }{2\max\{1,\beta'_J\} L}\frac{1}{|B_r|}\int_{B_r} |f(y)| \, dy\right)&\leq  \phi_{B_r}^-\left(\frac{\beta_J}{2\max\{1,\beta'_J\} L}\frac{1}{|B_r|}\int_{B_r}|f(y)|\,dy\right)+1\\ &\leq \beta'_J \frac{1}{|B_r|}\int_{B_r} \phi_{B_r}^-\big(\tfrac{1}{2\max\{1,\beta'_J\}L} |f(y)|\big) \, dy +1\\&\leq\frac{1}{|B_r|}\int_{B_r} \phi(y, |f(y)|) \, dy +1,
    \end{align*}
   where we used \ainc{1} in the last line. 
\end{proof}

\begin{remark}
    One could replace \eqref{eq:keyeq} with 
    \begin{equation*}
        \phi\left(x,\beta \frac{1}{|B|}\int_{B} |f| \, dx\right)\leq \frac{1}{\omega(B)}\int_{B} \phi(y, |f|)\omega(y) \, dy +1
    \end{equation*}
    in Lemma \ref{lem:wkeyest}. We can also observe that for Lemma \ref{lem:wkeyest} we only need for the weight to satisfy $\phi_{B}^-\left(\fint_B f\right)\lesssim \frac{1}{\omega(B)}$ for every $f\in L^{\varphi}(B,\omega)$ with 
   $\varrho_{L_\omega^\phi(B)}(f)\le 1.$
    This holds for example when the maximal operator is bounded in $L^{\varphi}(\Omega,\omega)$.
\end{remark}

\begin{lemma}\label{lem:RPest}
    Let $\omega \in A_p$ and $\phi\in \Phiw(B_r)$ satisfy \aonew{} and \ainc{p}, with $p \in (1, \infty)$. Then there exists $\beta=\beta(\beta_1,L,p,[\omega]_{A_p})>0$ such that
   \begin{equation*}
        \phi\left(x,\beta \int_{B_r} \frac{|f(y)|}{r |x-y|^{n-1}}\,dy\right)\leq \int_{B_r} 
\frac{\phi(y,|f(y)|)}{ r |x-y|^{n-1}} \,dy+1   \end{equation*}
for almost every $x\in {B_r}  $ and every $f\in L^{\varphi}(B_r,\omega)$ with $\varrho_{L_\omega^\phi({B_r})}(f)\le 1.$
\end{lemma}

\begin{proof}
    We follow the strategy in the proof of  \cite[Lemma~6.3.11]{HarH19}. Define annuli $A_k:=\{y \in {B_r} : 2^{-k}r\leq |x-y|<2^{1-k}r\}$ for $k\geq 0$. We cover ${B_r} $ with  $\{A_k\}$ and obtain 
    \begin{equation*}
        \int_{B_r}  \frac{|f(y)|}{2r|x-y|^{n-1}}\,dy\lesssim \sum_{k=1}^\infty 2^{-k} \fint_{B_{2^{1-k}r}(x)}\chi_{A_k}(y)|f(y)|\,dy.
    \end{equation*}

Let $\beta=\beta(\beta_1,L,p,[\omega]_{A_p})>0$ be from Lemma \ref{lem:wkeyest}.
Since $\sum_{k=0}^\infty 2^{-k-1}=1$, it follows that
\begin{align*}
    \phi\left(x,\beta_c \beta \int_{B_r}  \frac{|f(y)|}{2r|x-y|^{n-1}} \,dy \right) &\lesssim  \phi\left(x,\beta_c \beta \sum_{k=0}^\infty 2^{-k-1} \fint_{B_{2^{1-k}r}(x)}\chi_{A_k}(y)|f(y)| \,dy \right) \\
    &\leq \sum_{k=0}^\infty 2^{-k-1}\phi\left(x, \beta  \fint_{B_{2^{1-k}r}(x)}\chi_{A_k}(y)|f(y)| \,dy \right)
\end{align*}
for some $\beta_c=\beta_c(L)>0$,
by \cite[Corollary~2.2.2]{HarH19}. Here, Lemma~\ref{lem:wkeyest} yields that
\[
    \phi\left(x, \beta  \fint_{B_{2^{1-k}r}(x)}\chi_{A_k}(y)|f(y)| \,dy \right) 
     \le \fint_{B_{2^{1-k}r}(x)}\phi(y,\chi_{A_k}(y)|f(y)| )\,dy +1.
\]
Then we derive that
\begin{align*}
    \phi\left(x,\beta_c \beta \int_{B_r} \frac{|f(y)|}{2r|x-y|^{n-1}} \,dy \right) &\lesssim    \sum_{k=0}^\infty 2^{-k-1}\fint_{B_{2^{1-k}r}(x)}\phi(y,\chi_{A_k}(y)|f(y)| )\,dy +1 \\
    &\lesssim  \sum_{k=0}^\infty \frac{1}{2r(2^{1-k}r)^{n-1}}\int_{A_k}\phi(y,|f(y)| )\,dy +1 \\
    &\leq \int_{B_r} \frac{\phi(y,|f(y)| )}{2r|x-y|^{n-1}}\,dy +1. 
\end{align*} 
By \ainc{1}, we absorb the constant and the claim follows for suitable $\beta$.

\end{proof}

We need the following Sobolev-Poincar\'e type inequality and  Caccioppoli inequality in the weighted setting.
From Lemma~\ref{lem:property_A} (2), 
we note that $A_p \subset A_{p-\tilde{\epsilon}}$ for some small $\tilde{\epsilon} = \tilde{\epsilon}(n,p,[\omega]_{A_p})>0$. From now on, we set $\tilde{p} := p-\tilde\epsilon$.

\begin{proposition} \label{prop:wsobopoin}
Let $\omega \in A_p$ and  $\phi\in \Phiw(B_r)$  satisfy 
\aonew{}, 
\ainc{p}, and \adec{q} with $1<p\le q <\infty$, and 
  $v\in W^{1,1}(B_r)$ 
 with  $\varrho_{L_\omega^{\phi^{1/s}}(B_{r})}(\nabla v)\le M$ for some $M>0$ and $1\le s< \min\{\frac{p}{\tilde{p}}, \frac{n}{n-1}\}$. 
 Then we have
\begin{equation}\label{eq:wsobopoin1}
\frac{1}{\omega(B_r)}\int_{B_r}\phi\left(x, \frac{|v-(v)_{B_r}|}{r}\right) \, \omega (x) dx \le c\left\{\left(\frac{1}{\omega(B_r)}\int_{B_r}\phi (x, |\nabla v|)^{\frac1s} \, \omega (x)dx\right)^{s}+1\right\}
\end{equation}
for some $c=c(n,p,q,s, \beta_0, \beta_1, L, [\omega]_{A_p},M)>0$. 
%
\end{proposition}

\begin{proof}
From \cite[Lemma 8.2.1]{DieHHR11}, we note that for almost every $x \in B_r$, 
\[
|v(x) - (v)_{B_r}| \le c_1 \int_{B_r} \frac{|\nabla v(y)|}{|x-y|^{n-1}}\,dy
\]
for some $c_1=c_1(n)>0.$
Since $\phi $ satisfies \ainc{p} and $s\tilde{p}< p$, then $\phi^{\frac1{s\tilde{p}}}$  satisfies \ainc{1}, and hence $\phi^{\frac1{s\tilde{p}}}\in \Phiw$. Then we can apply Lemma \ref{lem:RPest}  with the function $\phi^{\frac1{s\tilde{p}}}$ and the constant $\tilde \beta := \frac{\beta}{c_1},$ in order to derive that
\begin{equation}\label{est:phivDv}
\begin{split}
\phi\bigg(x, \tilde\beta 
\frac{|v(x)-(v)_{B_r}|}{r} \bigg)
&\le \phi\bigg(x, \beta \int_{B_r}
\frac{|\nabla v(y)|}{r|x-y|^{n-1}} \,dy  \bigg)\\
&\le 
\bigg( \int_{B_r} \frac{f(y)}{r|x-y|^{n-1}}\,dy\bigg)^{s \tilde p}+1
\end{split}
\end{equation}
for almost every $x \in B_r$, where $f(y):=\phi(y, |\nabla v(y)|)^{\frac1{s\tilde p}}$.

If $\int_{B_r} f(y)\,dy=0$, then $|\nabla v|=0$ a.e., and then $v$ is constant. Therefore, $v=(v)_{B_r}$ a.e. and so \eqref{eq:wsobopoin1} holds.

Assume $\int_{B_r} f(y)\,dy>0$. 
Recall from the inequality~(2) in the proof of \cite{CF85} that 
\begin{equation}\label{ineq:gw}
\bigg( \frac{1}{\omega(B_r)} \int_{B_r} \bigg[ \int_{B_r} \frac{|g(y)|}{|x-y|^{n-1}}\,dy\bigg]^{k\gamma} \omega(x)dx \bigg)^{\frac1{k \gamma}} \le c r \bigg( \frac{1}{\omega(B_r)}  \int_{B_r} |g(x)|^\gamma \omega(x)dx \bigg)^{\frac1 \gamma}
\end{equation}
for any $g \in L^\gamma_\omega(B_r)$ with $\omega \in A_\gamma$, where $1\le k \le \frac{n}{n-1}+\delta$ for some $\delta = \delta(n,\gamma, [\omega]_{A_\gamma})>0.$
Note that $f \in L_\omega^{\tilde p}(B_r)$ from the assumption that $\varrho_{L_\omega^{\phi^{1/s}}(B_{r})}(\nabla v)\le M$ for some $M>0$. Then we apply the inequality \eqref{ineq:gw} with $\gamma:=\tilde p, k:=s$ and $g:=f$ to discover that 
 \[
\bigg( \frac{1}{\omega(B_r)} \int_{B_r} \bigg[ \int_{B_r} \frac{|f(y)|}{r|x-y|^{n-1}}\,dy\bigg]^{s\tilde p} \omega(x)dx \bigg)^{\frac1{s\tilde p}} 
\le c \bigg( \frac{1}{\omega(B_r)}  \int_{B_r} |f(y)|^{\tilde p} \omega(y)\,dy \bigg)^{\frac1{\tilde p}},
\]
which implies that 
 \[
 \begin{split}
\frac{1}{\omega(B_r)} \int_{B_r} \bigg( \int_{B_r} \frac{|f(y)|}{r|x-y|^{n-1}}\,dy\bigg)^{s \tilde p} \omega(x)dx
\le   c\bigg(\frac{1}{\omega(B_r)}  \int_{B_r} |f(y)|^{\tilde p} \omega(y)\,dy \bigg)^{s}.
\end{split}
\]

Therefore, inserting the previous inequality into \eqref{est:phivDv}, we arrive at
\[
\begin{split}
&\frac{1}{\omega(B_r)} \int_{B_r}\phi\bigg(x,\tilde\beta 
\frac{|v(x)-(v)_{B_r}|}{r} \bigg)\omega(x)dx\\
&\quad  \le \frac{1}{\omega(B_r)} \int_{B_r}\bigg( \int_{B_r} \frac{f(y)}{r|x-y|^{n-1}}\,dy\bigg)^{s \tilde p} \omega(x)dx+1\\
&\quad  \lesssim   \bigg(\frac{1}{\omega(B_r)}  \int_{B_r} |f(y)|^{\tilde p} \omega(y)\,dy \bigg)^{s}+1\\
    &\quad  \lesssim \bigg(\frac{1}{\omega(B_r)}\int_{B_r} \phi(y, |\nabla v|)^{\frac1s} \omega(y)\,dy\bigg)^{s} +1,
\end{split}
\]
which implies the desired estimate \eqref{eq:wsobopoin1} from \adec{q}.
\end{proof}

\begin{lemma}[Caccioppoli inequality]\label{lem:Caccio}
 Let $\omega$ be a weight
 and  $\phi\in \Phiw(\Omega)$  satisfy  \adec{q} with $q \in (1,\infty).$
 Setting $\psi(x,t):=\phi(x,t)\omega (x)$, let $u \in W_{w,\loc}^{1,\phi}(\Omega)$ be a local quasiminimizer of the $\psi$-energy in $\Omega$. Then we have 
\[
\int_{B_r}\phi(x, |\nabla u |)\omega (x) \,dx \lesssim  \int_{B_{2r}} \phi\bigg(x, \frac{|u-(u)_{B_{2r}}|}{2r}\bigg)\omega (x)\,dx
\] 
for every $B_{2r} \Subset \Omega$, where the implicit constant depends only on $n, K, q,$ and $L$.

\end{lemma}

\begin{proof}
Let $1\le s_1<s_2\le 2$ and $\eta \in C^\infty_0(B_{s_2 r})$ be such that $0\le \eta \le 1,\ \eta=1$ in $B_{s_1 r}$ and $|\nabla\eta|\le \frac{2}{(s_2-s_1)r}$. 
We use  $v : = u - \eta(u-(u)_{B_{2r}})$ as a test function in Definition~\ref{def:quasiminimizer}, in order to obtain 
\begin{equation}\label{ineq:quasimini}
\int_{B_{s_2 r}} \phi(x,|\nabla u |)\omega (x)\,dx \le K \int_{B_{s_2 r}} \phi(x,|\nabla v|)\omega (x)\,dx
\end{equation}
for some $K \ge 1$.
Here, we note that 
\[
|\nabla v| \le (1-\eta)|\nabla u |+|\nabla\eta|\,|u-(u)_{B_{2r}}| \le 2\max\bigg\{ (1-\eta)|\nabla u |, \frac{4|u-(u)_{B_{2r}}|}{2(s_2-s_1)r}\bigg\}.
\]
Then by \adec{q}, we derive that 
\[\begin{split}
\phi(x,|\nabla v|)& \le \phi(x, 2(1-\eta)|\nabla u |)   +  \phi \bigg( x, \frac{8|u-(u)_{B_{2r}}|}{2(s_2-s_1)r}\bigg)\\
& \le 2^q L \phi(x, (1-\eta)|\nabla u |)+ 8^q L \phi \bigg( x, \frac{|u-(u)_{B_{2r}}|}{2(s_2-s_1)r}\bigg). 
\end{split}\] 
Inserting this estimate into \eqref{ineq:quasimini}, we obtain 
\[\begin{split}
&\int_{B_{s_1 r}} \phi(x,|\nabla u |)\omega (x)\,dx  \le \int_{B_{s_2 r}} \phi(x,|\nabla u |)\omega (x)\,dx\\
& \le K \int_{B_{s_2 r}} \phi(x,|\nabla v|)\omega (x)\,dx\\
&  \le 2^q L K \int_{B_{s_2 r}} \phi(x, (1-\eta)|\nabla u |) \omega (x)\,dx + 8^q L K \int_{B_{s_2 r}}\phi \bigg( x, \frac{|u-(u)_{B_{2r}}|}{2(s_2-s_1)r}\bigg)\omega (x)\,dx.
\end{split}\]
Note that $\phi(x, (1-\eta)|\nabla u |) = \phi(x, 0) =0 $ in $B_{s_1r}$. Then it follows that 
\[\begin{split}
\int_{B_{s_1 r}} \phi(x,|\nabla u |)\omega (x)\,dx  
&  \le 2^q L K \int_{B_{s_2 r}\backslash B_{s_1 r}} \phi(x, |\nabla u |) \omega (x)\,dx\\
&\quad + 8^q L K \int_{B_{s_2 r}}\phi \bigg( x, \frac{|u-(u)_{B_{2r}}|}{2(s_2-s_1)r}\bigg)\omega (x)\,dx.
\end{split}\]
Using the hole-filling trick by adding $ 2^q L K \int_{B_{s_1r}} \phi(x, |\nabla u |) \omega (x)\,dx$ to both sides of the previous inequality, we derive that 
\[\begin{split}
\int_{B_{s_1 r}} \phi(x,|\nabla u |)\omega (x)\,dx   & \le \frac{2^q L K}{2^q L K +1} \int_{B_{s_2 r}} \phi(x, |\nabla u |) \omega (x)\,dx \\
&  \quad + \frac{8^q L K}{2^q L K +1} \int_{B_{2 r}}\phi \bigg( x, \frac{|u-(u)_{B_{2r}}|}{2(s_2-s_1)r}\bigg)\omega (x)\,dx. 
\end{split}\]

Consequently, since the function $\phi(x,t)\omega (x)$ is doubling for all $t>0$, we can apply a variant of the standard iteration lemma, Lemma~4.2 in \cite{HarHT17} with $$Z(\tau) : = \int_{B_{\tau r}} \phi(x,|\nabla u |)\omega (x)\,dx,\, X(\tau): =\frac{8^q L K}{2^q L K +1} \int_{B_{2 r}}\phi \bigg( x,\tau \frac{|u-(u)_{B_{2r}}|}{2r}\bigg) \omega (x)\,dx,$$ $t:=s_1$, $s:=s_2$, and $\theta := \frac{2^q L K}{2^q L K +1} \in [0,1) $, 
 in order to obtain 
\[
\int_{B_{ r}} \phi(x,|\nabla u |)\omega (x)\,dx  \lesssim
 \int_{B_{2 r}}\phi \bigg( x, \frac{|u-(u)_{B_{2r}}|}{2r}\bigg)\omega (x)\,dx.
\]


\end{proof}

Now we are ready to prove our main theorem, Theorem~\ref{mainthm}. 
It is sufficient to prove our main results under the assumption \adec{q}, since the argument in Lemma~\ref{lem:phi+t} allows us to extend them to the case \adeci{q}.

\begin{proof}[Proof of Theorem~\ref{mainthm}] 
Let $p>1$ be such that $\phi$ satisfies \ainc{p} and fix  $1\le s <\min\{ \frac{p}{\tilde{p}}, \frac{n}{n-1}\}$.

Choose a ball $B_{2r}$ such that $B_{2r} \Subset \Omega$ and $\varrho_{L_\omega^{\phi}(B_{2r})}(\nabla u )\le 1$. 
 From Caccioppoli inequality in Lemma~\ref{lem:Caccio},
 we obtain that 
\[
\frac{1}{\omega(B_r)}\int_{B_r}\phi(x, |\nabla u |)\omega (x) \,dx \lesssim \frac{1}{\omega(B_{2r})}\int_{B_{2r}} \phi\bigg(x, \frac{|u-(u)_{B_{2r}}|}{2r}\bigg)\omega (x)\,dx
\] 
after dividing both sides by $\omega (B_r)$  and  using the doubling property of $\omega$.
Note that $\varrho_{L_\omega^{\phi^{1/s}}(B_{2r})}(\nabla u )\le M$ for some $M=M(n,p,[\omega]_{A_p})>0$ by using Jensen's inequality.
Then the Sobolev-Poincar\'e type inequality \eqref{eq:wsobopoin1} in Proposition~\ref{prop:wsobopoin} yields  that
\[
\frac{1}{\omega (B_r)}\int_{B_r}\phi(x, |\nabla u |)\omega (x) \,dx 
\lesssim   \left(\frac{1}{\omega (B_{2r})}\int_{B_{2r}}\phi (x, |\nabla u |)^{\frac1s} \, \omega (x)dx\right)^{s}+1.
\]
Applying Lemma~\ref{lem:Gehring},
it follows that  $\phi(\cdot, |\nabla u |) \in L_\omega^{1+\epsilon}(B_r)$ 
with the estimates
\[\begin{split}
\bigg(\frac{1}{\omega (B_r)}\int_{B_r}\phi(x, |\nabla u |)^{1+\epsilon} \omega (x) \,dx \bigg)^{\frac1{1+\epsilon}}
& \lesssim \frac{1}{\omega (B_{2r})}\int_{B_{2r}}\phi (x, |\nabla u |) \, \omega (x)dx+1
\end{split}\]
for some $\epsilon = \epsilon(n,K,p,q,\beta_0,\beta_1, L,[\omega]_{A_p})>0$.
This completes the proof after a standard covering argument.
%
%

%
%
%
\end{proof}

\section*{Acknowledgment}
M. Lee was supported by the National Research Foundation of Korea(NRF)
grant funded by the Korea government(MSIT)(RS-2025-23525636). V. Hietanen was supported by the Finnish Cultural Foundation.
\bibliographystyle{amsplain}

\end{document}